\documentclass[oneside, 11pt]{amsart}

\usepackage{mathtools}
 \usepackage{amsmath, amscd}
  \usepackage{amsfonts}
  \usepackage{amssymb}
\usepackage{fullpage}
\usepackage{amsthm}
\usepackage{setspace}

\newtheorem{theorem}{Theorem}[section]

\newtheorem{prop/def}[theorem]{Proposition/Definition}
\newtheorem{corollary}[theorem]{Corollary}
\newtheorem{lemma}[theorem]{Lemma}
\theoremstyle{definition}

\newtheorem*{example}{Example}

\newtheorem*{remark}{Remark}

\newtheorem*{conclusion}{Conclusion}
\DeclareMathOperator{\End}{\mathrm{End}}
\DeclareMathOperator{\Gal}{\mathrm{Gal}}
\DeclareMathOperator{\Hom}{\mathrm{Hom}}
\DeclareMathOperator{\Qalg}{\mathbb{Q}-\mathrm{alg}}
\DeclareMathOperator{\ord}{\mathrm{ord}}
\DeclareMathOperator{\Lie}{\mathrm{Lie}}

\begin{document}
\title{A Deuring criterion for abelian varieties}
\author{Chris Blake}

\address{ D.P.M.M.S. \\ 
   Cambridge University \\
   Cambridge \\
  CB3 0WA \\
   U.K..}
\email{cb529@cam.ac.uk}


\begin{abstract}
We give several generalisations of the Deuring reduction criterion for elliptic curves to abelian varieties of higher dimension. In particular the Newton polygon of the reduction of an abelian variety \(A\) with complex multiplication by \(F\) at a prime above \(p\) can be explicitly computed in terms of the CM type of \(A\) and the decomposition group of a suitable prime above \(p\) in a Galois closure of \(F\). Conversely, it is shown that the formal isogeny type of the reduction of \(A\) at a prime above \(p\) places certain restrictions on the decomposition of \(p\) in \(F\). 
\end{abstract}

\maketitle

\section{Introduction} 

\let\thefootnote\relax\footnotetext{2000 \emph{Mathematics Subject Classification} 11G10, 11G15}

\let\thefootnote\relax\footnotetext{This research was supported by the Engineering and Physical Sciences Research Council}

In this paper we prove several generalisations of the Deuring reduction criterion for elliptic curves to abelian varieties of higher dimension. We fix an algebraic closure \(\overline{\mathbb{Q}}\) of \(\mathbb{Q}\) and let \(E \subset \overline{\mathbb{Q}}\) be a number field.  Let \(A / E \) be an abelian variety of dimension \(g\) with complex multiplication (CM) by a second number field \(F\) - i.e. \([F : \mathbb{Q}] = 2g\) and we have a  given homomorphism \(\iota: F \hookrightarrow \mathbb{Q} \otimes_{\mathbb{Z}} \End(A) \) (in particular these endomorphisms are defined over \(E\)). Then \(F\) is  necessarily a CM field, i.e a totally imaginary quadratic extension of a totally real field \(F^+\), which we can view as the fixed field of conjugation \(c \in \Gal(F/F^+)\).

If \(p\) is a prime number and \(\mathcal{P}\) is a prime of \(E\) above \(p\) at which \(A\) has good reduction we write \(A_{\mathcal{P}}\) for the reduction of \(A\) at \(\mathcal{P}\) - this is an abelian variety of dimension \(g\) over the finite field \(k = \mathcal{O}_E / \mathcal{P}\). When \(A\) is an elliptic curve the reduction \(A_{\mathcal{P}}\) can be either ordinary or supersingular, and the Deuring reduction criterion (c.f. \cite{Deuring}) relates the decomposition of \(p\) in \(F\) to the reduction type of \(A\) at \(\mathcal{P}\).

\begin{theorem}[Deuring] Let \(A/ E\) be an elliptic curve with complex multiplication by an imaginary quadratic field \(F\). Let \({\mathcal{P}}\) be any prime of \(E\) above \(p\) at which \(A\) has good reduction. Then \(A\) has supersingular reduction at \({\mathcal{P}}\) if and only if \(p\) is inert or ramified in \(F\), and ordinary reduction at \({\mathcal{P}}\) if and only if \(p\) is split in \(F\). 
 \end{theorem}

In higher dimensions the situation is altogether more complicated. In this paper, the appropriate generalisation of the reduction type of an elliptic curve is the Newton polygon (i.e. the formal isogeny type) of the reduced abelian variety \(A_{\mathcal{P}}\), which is said to be supersingular (respectively ordinary) if every slope of \(A_{\mathcal{P}}\) is equal to \(1/2\) (respectively either 0 or 1). Some results generalising the Deuring criterion to this setting have been obtained by Sugiyama \cite{Sugiyama}. For example, if \(p\) is unramified in \(F\) and every prime of \(F^+\) above \(p\) is inert in \(F\) then \(A_{\mathcal{P}}\) must be supersingular. Similarly, if \(p\) splits completely in \(F\) then \(A_{\mathcal{P}}\) must be ordinary. However, as the following example shows, the relationship between the decomposition of \(p\) in \(F\) and the the formal isogeny type of \(A_{\mathcal{P}}\) is not as simple as that described by Theorem 1.1,  even for abelian surfaces:

\begin{example} Let \(F = \mathbb{Q}(\zeta_8)\) and \(\Phi \subset \Hom_{\Qalg}(F, \overline{\mathbb{Q}})\) be any CM type which is induced from the imaginary quadratic subfield \(\mathbb{Q}(i)\).  Then the abelian variety \(A / \overline{\mathbb{Q}}\) with type \(\Phi\) is isogenous to \(A_0 \times A_0\), where \(A_0\) is an elliptic curve with complex multiplication by \(\mathbb{Q}(i)\). It follows from the classical Deuring criterion that the reduction of \(A\) at any prime \(\overline{\mathcal{P}}\) of \(\overline{\mathbb{Q}}\) above \(p\) is supersingular if \(p =2\) or \(p = 3\) (mod 4) and ordinary if \(p =1\) (mod 4). On the other hand if \(p = 3,5 \) (mod 8) then \(p\) decomposes as  \(p\mathcal{O}_F = \mathfrak{P} \mathfrak{P}^c\) in \(F\).  We conclude:
\begin{itemize}
\item If \(p = 3\) (mod 8) then \(p\mathcal{O}_F = \mathfrak{P} \mathfrak{P}^c\) and \(A\) has supersingular reduction at \(\overline{\mathcal{P}}\).
\item If \(p = 5\) (mod 8) then \(p\mathcal{O}_F = \mathfrak{P} \mathfrak{P}^c\) and \(A\) has ordinary reduction at \(\overline{\mathcal{P}}\).
\end{itemize}

 \end{example}

 In particular we note that \(A\) can have supersingular reduction at \(\mathcal{P}\) even if every prime of \(F^+\) above \(p\) splits in \(F\). Nonetheless it is possible to generalise Theorem 1.1 to higher dimensions. To ease notation let us now fix an embedding of \(F\) into \(\overline{\mathbb{Q}}\). If we let \(\widetilde{F}\) denote the Galois closure of \(F\) in \(\overline{\mathbb{Q}}\) then we can identify the CM type \(\Phi\) of \(A\), which is strictly speaking a subset of \(\Hom_{\Qalg}(F, \overline{\mathbb{Q}})\), with a subset of \(\Hom_{\Qalg}(F,  \widetilde{F})\). We will write \(G\) for \( \Gal( \widetilde{F}/ \mathbb{Q})\), \(H\) for \( \Gal( \widetilde{F}/F)\)
and \( \widetilde{\Phi}\) for the set of all lifts to \( \widetilde{F}\) of elements in \(\Phi\), i.e. \( \widetilde{\Phi} = \{g \in G : g|_F \in \Phi \}  \). Our first main result (c.f. \S3) asserts that the formal isogeny type of the reduction of \(A\) at a prime of good reduction can be computed in terms of purely group theoretic data as follows:

\begin{theorem} Let \(\mathcal{P}\) be a prime of \(E\) at which \(A\) has good reduction \(A_{\mathcal{P}}\).  Pick any prime \(\overline{\mathcal{P}}\) of \(\overline{\mathbb{Q}}\) above \(\mathcal{P}\) and let \( \widetilde{\mathfrak{P}}\) be the prime of \( \widetilde{F}\) below \(\overline{\mathcal{P}}\) and \(D = D_ { \widetilde{\mathfrak{P}}} \subset G\) be the decomposition group of \( \widetilde{\mathfrak{P}}\). Then the Newton polygon of \(A_{\mathcal{P}}\) has slopes \(\lambda_g = {|DgH \cap  \widetilde{\Phi}|}/{|DgH|}\) appearing with multiplicity \({|DgH|}/{|H|}\) as \(g\) ranges through a set of  representatives for double cosets \(D \setminus G  \; / H\). 
\end{theorem}

\begin{remark} Theorem 1.2 is in fact true as stated even if the complex multiplication is defined only after passing to a finite extension \(E^{\prime}\) of  \(E\).  This follows immediately from the fact that the Newton polygon of an abelian variety over finite field can be computed after passing to a finite extension.
\end{remark}

We briefly make some historical remarks on Theorem 1.2.  Zaystev \cite{Zaystev} has shown that if \(A\) has complex multiplication by the full ring of integers of \(F\) and if \(p\) is unramified in \(F\) then the \(BT_1\) group scheme \(A_{\mathcal{P}}[p]\) is determined up to geometric isomorphism by the CM type \(\Phi\) and an appropriate decomposition group in the Galois closure of \(F\). In its own way this generalises the elliptic curve case, and in low dimensions this has been made explicit by Goren (for abelian surfaces in work predating the general treatment of Zaystev, \cite{Goren}), and later Zaystev (for abelian threefolds, \cite{Zaystev}). Our Theorem 1.2 is an analogue of this result for the formal isogeny type of \(A_{\mathcal{P}}\) which is explicit in all dimensions, and which can handle the case when \(p\) is ramified in \(F\). The precise relationship between the Newton polygon of \(A_{\mathcal{P}}\) (which depends only on \(A_{\mathcal{P}}\) up to isogeny) and the isomorphism class of \(A_{\mathcal{P}}[p]\) is rather subtle, and we do not discuss it here (see for example,  \cite{Oort}).

As we saw from the work of Sugiyama, in some cases there is a more direct relationship between the decomposition of \(p\) in \(F\) and the formal isogeny type of the reduction of \(A\) at \(\mathcal{P}\).   
In \S4 we show that the the formal isogeny type of the reduction of \(A\) at \(\mathcal{P}\) places certain restrictions on the decomposition of \(p \) in \(F\), independent of the CM type. In particular we deduce:

\begin{theorem}Let \(A/E\) be an abelian variety of odd dimension \(g\) which has complex multiplication by  \(F\). Suppose there is a prime \(\mathcal{P}\) of \(E\) above \(p\) at which \(A\) has good supersingular reduction. Then there is at least one prime of \(F^{+}\) above \(p\) which is either inert or ramified in \(F\). 
\end{theorem}

In other words the situation for abelian surfaces we observed in Example 1, where \(A\) can have supersingular reduction at \(\mathcal{P}\) and yet every prime of \(F^+\) above \(p\) can split in \(F\), can be ruled out whenever the dimension of \(A\) is odd.

\section{Slopes of abelian varieties with complex multiplication}

Determining the formal isogeny type of the reduction of an abelian variety at a prime ideal is a purely local problem, and so we work in the analogous local situation. To this end let \(L\) be a finite extension of \(\mathbb{Q}_p\), \(\mathcal{O}_L\) its ring of integers, \(\mathcal{P}\) its maximal ideal, \(k\) the residue field of \(L\) and \(q = p^f\) the cardinality of \(k\). 
We let \({A}\) be an abelian variety over \(L\) of dimension \(g\) with good reduction \(A_{\mathcal{P}}\), and suppose \(A\) has complex multiplication by a CM field \(F\) of degree \(2g\), i.e. we have a  given homomorphism \(\iota: F \hookrightarrow \mathbb{Q} \otimes_{\mathbb{Z}} \End(A) \).

If \(C\) is any algebraically closed field of characteristic zero then a \(C\)-CM type for \(F\) is a set of representatives for the free action of \(\Gal(F/F^+)\) on \( \Hom_{\Qalg}(F, C)\), i.e. a collection of \(g\) embeddings of \(F\) into \(C\) such that no two are conjugate. If we fix an algebraic closure \(\overline{\mathbb{Q}}_p\) of \(\mathbb{Q}_p\) and an embedding \(j : L \hookrightarrow \overline{\mathbb{Q}}_p\) then the action of \(F\) on \(\Lie(A_{\overline{\mathbb{Q}}_p}) = \Lie(A)\otimes_{L, j}\overline{\mathbb{Q}}_p\) defines a \(\overline{\mathbb{Q}}_p\)-CM type \(\Phi \subset \Hom_{\Qalg}(F, \overline{\mathbb{Q}}_p)\) (this can be easily reduced to the classical case over \(\mathbb{C}\), where it follows from the complex uniformisation of abelian varieties). We now review how the Newton polygon of \({A}_{\mathcal{P}}\) can be expressed in terms of this CM type. 

First recall that the slopes of \(A_{\mathcal{P}}\) are the (suitably normalised) \(p\)-adic valuations of the eigenvalues of the \(q\)-Frobenius endomorphism of \(A_{\mathcal{P}}\).  More precisely, if \(G \in \mathbb{Z}[T]\) is the characteristic polynomial of \(q\)-Frobenius endomorphism of \(A_{\mathcal{P}}\) then we can write 
\( G(T) = \prod_{i=1}^{2g} (T - \alpha_i) \)
for some \(\alpha_i \in \overline{\mathbb{Q}}_p\). The slopes of \(A_{\mathcal{P}}\) are by definition the rational numbers \[\lambda_i = \ord_p(\alpha_i)/f\] (here \(\ord_p\) is normalised so that \(\ord_p(p)=1 \)), and the data of the slopes of \(A_{\mathcal{P}}\) can be represented visually by a Newton polygon. It follows from this definition that the Newton polygon of \(A_{\mathcal{P}}\) depends only the abelian variety \(A_{\mathcal{P}}\) up to isogeny, and is invariant under replacing \(k\) by some finite extension \(k^{\prime}\). In fact, by the Dieudonn\'e-Manin classification, this Newton polygon precisely classifies the associated \(p\)-divisible group \(A_{\mathcal{P}}(p)\) up to isogeny over an algebraic closure of \(k\) (c.f. \cite{Demazure}). 

Since \(A\) has complex multiplication we can actually compute the Newton polygon of \(A_{\mathcal{P}}\) in characteristic 0. Recall there is always an element \(\pi\) of \(\mathcal{O}_F\) such that \(\iota(\pi)\) reduces to the \(q\)-Frobenius endomorphism of \(A_{\mathcal{P}}\) (the proof in \S7 of \cite{Serre-Tate}  works in this setting). Then the characteristic polynomial of the \(q\)-Frobenius endomorphism of \(A_{\mathcal{P}}\) is equal to the characteristic polynomial of \(\pi \in \mathcal{O}_F\) and so the slopes of \(A_{\mathcal{P}}\) are just the rational numbers \[\lambda_w =  \ord_w(\pi) / \ord_w(q) \] appearing with multiplicity \([F_w : \mathbb{Q}_p]\) as \(w\) ranges through all places of \(F\) above \(p\) (note here the choice of normalisation of \(\ord_w\) cancels).
This can then be reinterpreted in terms of the associated \(p\)-divisible group \(\mathcal{G} = \mathcal{A}(p) \), where \(\mathcal{A}/ \mathcal{O}_L\) denotes the N\'eron model of \(A\). 
 Indeed the inclusion \(\iota: F \hookrightarrow \mathbb{Q} \otimes_{\mathbb{Z}} \End(A) \) induces an inclusion \(\iota : \mathbb{Q}_p \otimes_{\mathbb{Q}} F \hookrightarrow \mathbb{Q}_p \otimes_{\mathbb{Z}_p} \End(\mathcal{G}) \) and corresponding to the decomposition \(\mathbb{Q}_p \otimes_{\mathbb{Q}} F  \simeq \prod_{w|p} F_w\) (where \(w\) runs over places of \(F\) above \(p\)) there is a decomposition 
\[ \mathcal{G} \sim \prod_{w|p} \mathcal{G}_w\]
of \(p\)-divisible groups up to isogeny. Let  \(h_w\) and \(d_w\) respectively denote the height and dimension of \(\mathcal{G}_w\) -  we will refer to the ratio \(d_w/h_w\) as the slope of \(\mathcal{G}_w\). Each \(\mathcal{G}_w\) is a \(p\)-divisible group equipped (via \(\iota)\) with an action of the local field \(K = F_w\) over the which the rational Tate module \(V\mathcal{G}_w\) is free of rank 1. This is precisely what it means to say that each \(\mathcal{G}_w\) is a \(p\)-divisible group with complex multiplication by the local field \(F_w\), and the following key result of Tate enables us to relate the slopes of the \(\mathcal{G}_w\) to the slopes of \(A_{\mathcal{P}}\).

\begin{theorem}[Shimura-Taniyama formula for \(p\)-divisible groups, \cite{Tate}, \S5] Let \(\mathcal{G}\) be a \(p\)-divisible group over \(\mathcal{O}_L\) of dimension \(d\) and height \(h\), and with complex multiplication by a local field \(K\). If there is an element \(\pi \in \mathcal{O}_K\) lifting the action of \(q\)-Frobenius  then \({\ord_K(\pi)}/{\ord_K(q)} = {d}/{h} \). 

\end{theorem}

 The Shimura-Taniyama formula (applied to each \(\mathcal{G}_w\) in turn) tells us that the slopes of \(A_{\mathcal{P}}\) are precisely the slopes of the \(\mathcal{G}_w\) taken with multiplicity \([F_w : \mathbb{Q}_p] = h_w\), i.e. 
  \[\lambda_w = \ord_w(\pi) / \ord_w(q) = d_w/h_w\] It is now easy to express each \(\lambda_w\) in terms of the \(\overline{\mathbb{Q}}_p\)-CM type \(\Phi\). Indeed if \(\Phi_w\) denotes the subset of \(\Phi\) consisting of embeddings which induce the place \(w\) on \(F\) then the dimension of \(\mathcal{G}_w\) is  \( |\Phi_w|\) (c.f. \cite{Tate}, \S 5). In summary:
 
\begin{conclusion}  The slopes of \(A_{\mathcal{P}}\) are precisely the slopes of the \(p\)-divisible groups \(\mathcal{G}_w\), i.e. the numbers \(\lambda_w = d_w / h_w = |\Phi_w| / [F_w : \mathbb{Q}_p]\) taken with multiplicity \(h_w = [F_w : \mathbb{Q}_p] \) as \(w\) ranges through all places of \(F\) above \(p\). 
\end{conclusion}

\section{An analogue of Deuring's Criterion}

We now place ourselves in the global situation of \S1.  We have two number fields \(E, F \subset \overline{\mathbb{Q}}\) and an abelian variety \(A / E\) of dimension \(g\) and with complex multiplication by \(F\). The CM type of \(A\) is the subset \(\Phi \subset  \Hom_{\Qalg}(F, \overline{\mathbb{Q}})\) of size \(g\) defined by the action of \(F\) on \(\Lie(A_{\overline{\mathbb{Q}}}) = \Lie(A) \otimes_{E} \overline{\mathbb{Q}}\), and it again it follows from the complex analytic theory that no two of these embeddings are conjugate. 

If we let \( \widetilde{F}\) be the Galois closure of \(F\) in \(\overline{\mathbb{Q}}\) then the CM type \(\Phi\) can be identified with a subset of \(\Hom_{\Qalg}(F,  \widetilde{F})\). 
We will write \(G\) for \( \Gal( \widetilde{F}/ \mathbb{Q})\), \(H\) for \( \Gal( \widetilde{F}/F)\)
and \( \widetilde{\Phi}\) for the set of all lifts to \( \widetilde{F}\) of elements in \(\Phi\), i.e. \(
 \widetilde{\Phi} = \{g \in G : g|_F \in \Phi \} \subset G \). The formal isogeny type of the reduction of \(A\) at a prime of good reduction can be computed in terms of purely group theoretic data as follows:

\begin{theorem} Let \(\mathcal{P}\) be a prime of \(E\) at which \(A\) has good reduction \(A_{\mathcal{P}}\).  Pick any prime \(\overline{\mathcal{P}}\) of \(\overline{\mathbb{Q}}\) above \(\mathcal{P}\) and let \( \widetilde{\mathfrak{P}}\) be the prime of \( \widetilde{F}\) below \(\overline{\mathcal{P}}\) and \(D = D_ { \widetilde{\mathfrak{P}}} \subset G\) be the decomposition group of \( \widetilde{\mathfrak{P}}\). Then the slopes of \(A_{\mathcal{P}}\) can be computed as the numbers \(\lambda_g = {  |DgH \cap  \widetilde{\Phi} | }/{ |DgH|}\) appearing with multiplicity \({ |DgH|}/{|H|}\) as \(g\) ranges through a set of  representatives for double cosets \(D \setminus G  \; / H\). 
\end{theorem}

\begin{proof}  Let \(\overline{w}\) be the place of  \(\overline{\mathbb{Q}}\) corresponding to \(\overline{\mathcal{P}}\) and \( \widetilde{w}\) the place of \( \widetilde{F}\) below \(\overline{w}\).  If \(L\) denotes the completion of \(E\) at \(\mathcal{P}\) then we can consider \(A\) as an abelian variety over \(L\) and a choice of embedding \(\iota : \overline{\mathbb{Q}} \hookrightarrow \overline{\mathbb{Q}}_p\) inducing the place \(\overline{w}\) on \(\overline{\mathbb{Q}}\) induces an inclusion \(j : L \hookrightarrow \overline{\mathbb{Q}}_p \) with respect to which \(A\) has \(\overline{\mathbb{Q}}_p\)-CM type \(\iota \circ \Phi \subset \Hom_{\Qalg}(F, \overline{\mathbb{Q}}_p)\).
The associated \(p\)-divisible group \(\mathcal{G} = \mathcal{A}(p)\) is isogenous to a product
\[ \mathcal{G} \sim \prod_{w|p} \mathcal{G}_w\]
where each \(\mathcal{G}_w\) is a \(p\)-divisible group over \(\mathcal{O}_L\) of height \( h_w = [F_w : \mathbb{Q}_p]\) and dimension \(d_w\) equal to the number of embeddings in the \(\overline{\mathbb{Q}}_p\)-CM type \(\iota \circ \Phi\) which induce the place \(w\) of \(F\). The slopes of \(A_{\mathcal{P}}\) are then the ratios \(\lambda_w = d_w/h_w\) appearing with multiplicity \(h_w\) (c.f. \S2). 

Places \(w\) of \(F\) dividing \(p\) correspond to prime ideals \(\mathfrak{P}\) of \(\mathcal{O}_F\) above \(p\), and since we have fixed a prime \( \widetilde{\mathfrak{P}}\) of \( \widetilde{F}\) above \(p\) we can set up a correspondence between primes \(\mathfrak{P}\) of \(\mathcal{O}_F\) above \(p\) and double cosets \(D \setminus G \; / H\). A prime \(\mathfrak{P}\) of \(F\) corresponds to the double coset \(DgH\) if there is a prime of \( \widetilde{F}\) above \(\mathfrak{P}\) which is mapped to our chosen prime \( \widetilde{\mathfrak{P}}\) of \( \widetilde{F}\) by \(g \in G\). We can also think of this correspondence directly in terms of places. Firstly, we can use the map \(\iota :  \overline{\mathbb{Q}} \hookrightarrow \overline{\mathbb{Q}}_p\) to identify the sets 
\( G =  \Hom_{\Qalg}( \widetilde{F} ,  \widetilde{F}) \) and  \(\Hom_{\Qalg}( \widetilde{F}, \overline{\mathbb{Q}}_p) \). 
With these identifications made, if \(w\) is a place of \(F\) associated to the double coset \(DgH\) then this double coset \(DgH\) simply corresponds to the subset of \(\Hom_{\Qalg}( \widetilde{F}, \overline{\mathbb{Q}}_p) \) consisting of all embeddings \( \widetilde{F} \hookrightarrow \overline{\mathbb{Q}}_p\) which induce the place \(w\) on \(F\).  

Consequently the height \(h_w\) of \(\mathcal{G}_w\) can be computed as \(h_w = [F_w: \mathbb{Q}_p] =  |DgH|  / |H| \). Moreover the intersection \(DgH \cap  \widetilde{\Phi} \subset G\) corresponds to all the embeddings \( \widetilde{F} \hookrightarrow \overline{\mathbb{Q}}_p\) which induce the place \(w\) on \(F\) and which restrict to elements of the \(\overline{\mathbb{Q}}_p\)-CM type \(\iota \circ \Phi\) of \(A / L\).  Hence the dimension \(d_w\) can be computed as \(|DgH \cap \widetilde{\Phi} | / |H| \). 
\end{proof}

It is an easy exercise to see that this reproduces the classical Deuring criterion in the case when \(A\) is an elliptic curve. More interestingly we illustrate the above theorem by reproducing the following example of Goren (c.f. \cite{Goren}, Theorem 2).

\begin{example} Let \(F\) be a non-cyclic primitive quartic CM field and \(\Phi\) a CM type for \(F\). Then (c.f. \cite{ST}, II, 8.4) we have \(Gal( \widetilde{F} / \mathbb{Q}) \simeq D_8 = \langle x,y \; | \; x^2 = 1 = y^4, xyx = y^3 \rangle = \{1, y, y^2, y^3, x, xy, xy^2, xy^3\}\) and we may assume that \(F\) is the fixed field of the subgroup \(\{1,x\}\) and that \( \widetilde{\Phi} =  \{1,x, y, xy^3\} \). The maximal totally real subfield \(F^+\) is then fixed by \(\{1,x,y^2, xy^2\}\). Let \(A / E\) be an abelian variety with CM by \(F\) with type \(\Phi\) and let \(\mathcal{P}\) be a prime of \(E\) of good reduction for \(A\). We fix a prime \(\overline{\mathcal{P}}\) of \(\overline{\mathbb{Q}}\) above \(\mathcal{P}\) and let \( \widetilde{\mathfrak{P}}\) be the associated prime of \( \widetilde{F}\)  and \(\mathfrak{P}\) that of \(F\).

\begin{itemize}\item If \(D = \{1, xy\}\) then \(p\) splits in \(F\) as \(p\mathcal{O}_F= \mathfrak{P} \mathfrak{P}^c\). In the notation of the previous discussion these two primes correspond to the two double cosets \(DH = \{1, xy, x, y^3\}\) and \(Dy^2H = \{y^2, xy^3, y, xy^2\}\).  Then \(DH \cap  \widetilde{\Phi} = \{1,x\}\) and \(Dy^2H\cap  \widetilde{\Phi} = \{y, xy^3\}\). So we see that \(\mathcal{G} \sim \mathcal{G}_w \times \mathcal{G}_{w^c}\) where each component has height \(2\) and dimension \(1\). Hence \(A\) has supersingular reduction at \(\mathcal{P}\).  
\item If on the other hand \(D = \{1, xy^3\}\) then we still have \(p\) splitting as \(p\mathcal{O}_F = \mathfrak{P} \mathfrak{P}^c\). Now these two prime ideals correspond to the two double cosets \(DH = \{1, xy^3, x, y\}\) and \(Dy^2H = \{y^2, xy, xy^2, y^3\}\).  Here \(DH \cap  \widetilde{\Phi} = \{1,x, y, xy^3\}\) and \(Dy^2H\cap  \widetilde{\Phi} = \emptyset \) and so in this case  \(\mathcal{G} \sim \mathcal{G}_w \times \mathcal{G}_{w^c}\), but now \(\mathcal{G}_w\) has dimension \(2\) while \(\mathcal{G}_{w^c}\) has dimension \(0\). Hence \(A\) has ordinary reduction at \(\mathcal{P}\). 
\end{itemize}

\end{example}

\section{Results independent of the type} 

We saw in the previous section that if \(A/E\) is an abelian variety with complex multiplication by \(F\) then the formal isogeny type of the reduction of \(A\) at a prime \(\mathcal{P}\) of good reduction is completely determined by the interaction between the CM type of \(A\) and a suitable decomposition group in the Galois closure of \(F\). In this section we explore what can be said in general about the relationship between the decomposition of \(p\) in \(F\) and the formal isogeny type of the reduction of \(A\) at  a prime of good reduction, independent of the CM type. 
Let \(L\) denote the completion of \(E\) at the discrete valuation corresponding to \(\mathcal{P}\), and let \(\mathcal{G} / \mathcal{O}_L\) be the \(p\)-divisible group attached to \(A\). Then recall there is a decomposition 
\[ \mathcal{G} \sim \prod_{w|p} \mathcal{G}_w\]
of \(p\)-divisible groups up to isogeny, where \(w\) runs over the places of \(F\) above \(p\). The slopes of \(A_{\mathcal{P}}\) are the same as the slopes of the \(\mathcal{G}_w\), i.e the numbers \(\lambda_w = d_w/h_w\) appearing with multiplicity \(h_w\). 

We can better understand these slopes by considering the maximal totally real subfield \(F^+ \subset F\). 
Indeed the inclusion \(\iota :  \mathbb{Q}_p \otimes_{\mathbb{Q}} F^+ \rightarrow \mathbb{Q}_p \otimes_{\mathbb{Z}_p} \End(\mathcal{G}) \) also induces a decomposition
\[ \mathcal{G} \sim  \prod_{v|p} \mathcal{G}_v \]
where \(v\) runs over the places of \(F^+\) dividing \(p\). Here each \(\mathcal{G}_v\) is a \(p\)--divisible group equipped with an action of the local field \(F_v^+\) over the which the rational Tate module \(V\mathcal{G}_v\) is free of rank 2. Each \(\mathcal{G}_v\) further decomposes as 
\[ \mathcal{G}_v \sim \prod_{w|v} \mathcal{G}_w \]
recovering our earlier decomposition in terms of places \(w\) of \(F\).  
 The point is that the height \(h_v\) and dimension \(d_v\) of these groups are now independent of the \(\overline{\mathbb{Q}}_p\)-CM type \(\Phi\) (defined with respect to any embedding \(j: L \hookrightarrow \overline{\mathbb{Q}}_p\)). Indeed \(h_v = 2[F_v^+ : \mathbb{Q}_p] \) and since the embeddings of \(F\) into \(\overline{\mathbb{Q}}_p\) appearing in the \(\overline{\mathbb{Q}}_p\)-CM type all restrict to distinct embeddings of \(F^+\) we see that \(d_v = [F^+_v : \mathbb{Q}_p]\). In particular we see that the ratio \(d_v/h_v \) is equal to \(1/2\) for all \(v\). There are two cases to deal with: 

\begin{itemize}
\item If there is a unique place \(w\) of \(F\) above the place \(v\) of \(F^+\) then \(\mathcal{G}_v = \mathcal{G}_w\) and so we see \(\lambda_w = {1}/{2}\)
\item If we have two places \(w\) and \(w^c\) lying above \(v\) then we have a corresponding decomposition \(\mathcal{G}_v \sim\mathcal{G}_w \times \mathcal{G}_{w^c}\) into \(p\)-divisible groups of equal height \( [F^{+}_v : \mathbb{Q}_p] = [F_w : \mathbb{Q}_p] = [F_{w^c} : \mathbb{Q}_p] \).  It follows that the two slopes are related by \(\lambda_w = 1 - \lambda_{w^c}\) 
\end{itemize}

\subsection*{Supersingular reduction}

Recall that \(A_{\mathcal{P}}\) is supersingular if every slope of \(A_{\mathcal{P}}\) is equal to \(1/2\). Our methods enable us to generalise a result of Sugiyama (described in \S1 of this paper, c.f. \cite{Sugiyama} Theorem 1.1) to the case where \(p\) may be ramified in \(F\). 

\begin{lemma} Let \(A/E\) be an abelian variety with complex multiplication by  \(F\), and let \(\mathcal{P}\) be any prime of \(E\) above \(p\) at which \(A\) has good reduction. If all primes of \(F^{+}\) above \(p\) are inert or ramified in \(F\) then \(A\) has supersingular reduction at \({\mathcal{P}}\). 

\begin{proof} In this case there is a unique place \(w\) of \(F\) lying above any place \(v\) of \(F^+\) above \(p\). Consequently each \(\mathcal{G}_w = \mathcal{G}_v\) has slope \({1}/{2}\) and so \(A\) has supersingular reduction. 
\end{proof}
\end{lemma}

As we observed in Examples 1 and 2, even for abelian surfaces the converse to this result fails. Indeed we saw that \(A_{\mathcal{P}}\) can be supersingular and yet every prime of \(F^{+}\) above \(p\) can split in \(F\). However, we can at least prove the following partial result 

\begin{theorem} Let \(A/E\) be an abelian variety with complex multiplication by \(F\). Suppose there is a prime \(\mathcal{P}\) of \(E\) above \(p\) at which \(A\) has good supersingular reduction.  Then any prime \(\mathfrak{p}\) of \(F^{+}\) above \(p\) with \(e(\mathfrak{p} | p)f(\mathfrak{p} |p)\) odd must be either inert or ramified in \(F\). 
\begin{proof}  Let \(\mathfrak{p}\) be a prime of \(F^+\) above \(p\) which splits in \(F\). If \(v\) is the place of \(F^+\) associated to \(\mathfrak{p}\) then there are two distinct places \(w\) and \(w^c\) of \(F\) above \(v\). So we have a decomposition \(\mathcal{G}_v \sim\mathcal{G}_w \times \mathcal{G}_{w^c}\) and since \(A_{\mathcal{P}}\) is supersingular each component must have slope \(1/2\). But each component has height \(h_w = h_{w^c} = [F^+_v : \mathbb{Q}_p]\) and so in order for them to have slope \(1/2\) we need  \( [F^+_v : \mathbb{Q}_p] = e(\mathfrak{p} | p)f(\mathfrak{p} |p)\) to be even.  \end{proof} \end{theorem}

In particular we deduce that the situation we noticed for abelian surfaces (where \(A\) can have supersingular reduction at a prime above \(p\) and yet every prime of \(F^{+}\) above \(p\) can split in \(F\)) cannot occur whenever the dimension of \(A\) is odd. 

\begin{corollary}Let \(A/E\) be an abelian variety of odd dimension \(g\) which has complex multiplication by  \(F\). Suppose there is a prime \(\mathcal{P}\) of \(E\) above \(p\) at which \(A\) has good supersingular reduction. Then there is at least one prime of \(F^{+}\) above \(p\) which is either inert or ramified in \(F\). 
\begin{proof} Since 
\[ g = [F^+ : \mathbb{Q} ] = \sum_{\mathfrak{p} |p}  e(\mathfrak{p} | p)f(\mathfrak{p} |p) \]
at least one of the terms \( e(\mathfrak{p} | p)f(\mathfrak{p} |p)\) must be odd if \(g\) is. We may then apply Theorem 4.2. 
\end{proof}
\end{corollary}

In the case when \(F^+\) is moreover Galois over \(\mathbb{Q}\) we can completely classify supersingular reduction in terms of the decomposition of \(p\) in \(F\). 
\begin{corollary} Let \(A/E\) be an abelian variety of odd dimension \(g\) which has complex multiplication by  \(F\), and let \(\mathcal{P}\) be any prime of \(E\) above \(p\) at which \(A\) has good reduction.  Assume further that \(F^+\) is Galois over \(\mathbb{Q}\). Then \(A\) has supersingular reduction at \(\mathcal{P}\) if and only if every prime of \(F^{+}\) above \(p\) is either inert or ramified in \(F\). 
\begin{proof} In this case if \(\mathfrak{p}\) is a prime of \(F^+\) above \(p\) then we know that \(e(\mathfrak{p} | p)f(\mathfrak{p} |p)\) divides \(g =[F^+ :\mathbb{Q}]\) and so must be odd if \(g\) is. The result then follows from Theorem 4.2 and Lemma 4.1. 
\end{proof}
\end{corollary}

\subsection*{Ordinary reduction}
Recall \(A_{\mathcal{P}}\) is ordinary if every slope of \(A_{\mathcal{P}}\) is either \(0\) or \(1\). Sugiyama has also shown that if \(p\) splits completely in \(F\) then \(A\) has ordinary reduction at \(\mathcal{P}\) (\cite{Sugiyama}, Theorem 1.2).  
Even for abelian surfaces the converse to this result fails - indeed we saw that in this case we could have \(p\mathcal{O}_F = \mathfrak{P}\mathfrak{P}^c\) and \(A_{\mathcal{P}}\) can be ordinary. We can at least say

\begin{theorem}  Let \(A/E\) be an abelian variety with complex multiplication by \(F\). Suppose there is a prime \(\mathcal{P}\) of \(E\) above \(p\) at which \(A\) has good ordinary reduction. Then every prime of \(F^+\) above \(p\) must split in \(F\). 
\begin{proof} If not then there is a place \(v\) of \(F^+\) above \(p\) such that there is a unique place \(w\) of \(F\) above \(v\). In this case the corresponding \(p\)-divisible group \(\mathcal{G}_w = \mathcal{G}_v\) will have slope \(1/2\) and so \(A_{\mathcal{P}}\) cannot be ordinary. 
\end{proof}
\end{theorem}

\textsl{Acknowledgements}: This paper is the result of research carried out whilst studying for a PhD under the supervision of Tony Scholl, to whom I am grateful for suggesting the problem and for his guidance throughout. I would also like to thank James Newton for helpful comments on an earlier draft of this paper.


\begin{thebibliography}{9}



\bibitem{Demazure} M. Demazure, \emph{Lectures on p-divisible groups}, Lect. Notes Math. 302, Springer-Verlag, 1972.

\bibitem{Deuring} M. Deuring, \emph{Die Typen der Multiplikatorenringe elliptischer Funktionenk\"{o}rper} 
Abh. Math. Sem. Hansischen Univ. 14, (1941). 197Ð272.


\bibitem{Goren} E. Goren, \emph{On certain reduction problems concerning abelian surfaces}, Manuscripta Math., 94:33Ð43, 1997.

\bibitem{Oort} F. Oort, \emph{Simple \(p\)-kernels of \(p\)-divisible groups}, Adv. Math. 198 (2005), no. 1, 275Ð310.

\bibitem{Serre-Tate} J.-P. Serre, J. Tate, \emph{Good reduction of abelian varieties}, Ann. of Math. (2) 88 1968 492Ð517. 
 
 \bibitem{ST} G. Shimura, Y. Taniyama, \emph{Complex multiplication of abelian varieties and its applications to number theory}, 
Publications of the Mathematical Society of Japan, 6 The Mathematical Society of Japan, Tokyo 1961 

 \bibitem{Sugiyama} K. Sugiyama, \emph{On a generalization of Deuring's results},  arXiv [math.NT],1212.5883v5.
  
\bibitem{pdgroups} J. Tate, \emph{\(p\)-divisible groups}, 1967 Proc. Conf. Local Fields (Driebergen, 1966) pp. 158Ð183 Springer, Berlin.

\bibitem{Tate} J. Tate, \emph{Classes d'isog\'enie des vari\'et\'es ab\'eliennes sur un corps fini}, S\'eminaire Bourbaki 11 (1968-1969): 95-110.
 
\bibitem{Zaystev} A. Zaytsev, \emph{Generalization of Deuring reduction theorem},  J. Algebra 392 (2013), 97Ð114.


\end{thebibliography}
\end{document}